\documentclass[a4paper]{amsart}
\usepackage{mainstyle}

\usepackage{stackengine}

\makeatletter
\newcommand{\oset}[3][0ex]{%
  \mathrel{\mathop{#3}\limits^{
    \vbox to#1{\kern-2\ex@
    \hbox{$\scriptstyle#2$}\vss}}}}
\makeatother
\newcommand{\stab}{\mathrm{stab}}

\title{The Torsion of Automorphisms of Nilpotent Spaces}

\author[Sacha Goldman]{Sacha Goldman${}^1$} 
\address{University of Toronto, Department of Mathematics, 40 St. George St., Toronto, Ontario, Canada, M5S 2E4}%
\email{sacha.goldman@mail.utoronto.ca}
\thanks{${}^1$ Supported by an NSERC Canada Graduate Research Scholarship --- Master's and an Ontario Graduate Scholarship}
\begin{document}

\maketitle

\begin{abstract}
We reprise a $K_1$-valued refinement of Whitehead torsion originally studied by Gersten. We use this Gersten torsion to show that for nilpotent spaces with infinite fundamental group, any self-equivalence which acts as the identity on the fundamental group has vanishing Whitehead torsion. We find two applications of our vanishing result. First, we provide many examples of spaces with infinitely many simple structures. Second, we conclude that the group of homotopy classes of simple self-equivalences of a connected nilpotent space that act as the identity on the fundamental group is commensurable to an arithmetic group, building on a theorem of Sullivan. We also give a corrected version of Sullivan's proof as an appendix.  	
\end{abstract}

\tableofcontents

\section*{Introduction}

In 1950, Whitehead \cite{Whitehead1950} gave a complete obstruction to an equivalence between finite CW-complexes being simple. This obstruction is complete in that it is always realized by an equivalence $Y \to X$. In practice, we are often interested in studying self-equivalences $X \to X$, rather than all equivalences. When we restrict from equivalence to self-equivalences, remarkably little is known about the possible values of the Whitehead torsion. 

\begin{maintheorem}\label{MainVanishingThm}
	Suppose that $X$ is a connected based finite nilpotent CW-complex with infinite $\pi_1(X,x)$. If $f \colon X \to X$ is a based homotopy self-equivalence that acts as the identity on $\pi_1(X,x)$, then $f$ is simple.
\end{maintheorem}

Our \cref{MainVanishingThm} is inspired by a result of Mislin \cite[Theorem B]{Mislin1976} stating that a finitely dominated nilpotent CW-complex is always finite. The proof of \cref{MainVanishingThm} will rely on a new interpretation of Gersten's \cite{Gersten1967} $K_1$-valued refinement of the Whitehead torsion. 

We have two applications of \cref{MainVanishingThm}. Our first application is to better understand the difference between the relations of equivalence and simple homotopy equivalence.

\begin{maintheorem}\label{MainStructureThm}
	Suppose $X$ is a connected finite nilpotent CW-complex with infinite Whitehead group, infinite fundamental group, and such that the fundamental group of $X$ only has finitely many automorphisms. Then there are infinitely many finite CW-complexes with homotopy equivalent to $X$, none of which are simple homotopy equivalent.
\end{maintheorem}

In \cref{MainCorollariesSection} we provide some examples of homotopy types satisfying \cref{MainStructureThm}, $S^1 \times L(p,q_1, \dots, q_n)$ for $p \neq 2,3$, $S^1 \times S^3/Q_8$ which is an example with non-abelian fundamental group, and the Lie group $U(n)/\langle \zeta \rangle$ for an $n^\text{th}$ root of unity $\zeta$ and with $n \neq 1,2,3,6$.  For the first example of the product of a circle with a lens space, Nagy--Nicholson--Powell \cite[Theorem A]{Nagy2023} show the stronger statement that there are infinitely many h-cobordant manifolds with this homotopy type which are not simple homotopy equivalent. Combining their techniques with the proof of \cref{MainStructureThm} one can find many more infinite families of h-cobordant manifolds which are not simple homotopy equivalent.

Our other application of \cref{MainVanishingThm} is to study finiteness properties of the group of simple self-equivalences. Bustamante--Krannich--Kupers \cite[Theorem 2.2]{Bustamante2023} prove that this group is commensurable up to finite kernel to an arithmetic when $X$ has finite fundamental group. We analyze the case of $X$ with infinite fundamental group. The combination of our results yields the following theorem. 

\begin{maintheorem}\label{MainArithmeticityThm}
	If $X$ is a connected based finite nilpotent CW-complex, the group of based self-equivalences of $X$ up to based homotopy that fix $\pi_1(X,x)$ and which are simple is commensurable to an arithmetic group.
\end{maintheorem}

A result of Borel--Serre \cite{Borel1973} gives that groups commensurable to arithmetic groups are of finite type, that is, have finitely many cells in every dimension.

\cref{MainArithmeticityThm} is built upon a result of Sullivan \cite[Theorem 6.1]{Sullivan1977} who shows that for a connected based finite nilpotent CW-complex there is a map with finite kernel from the group of self-equivalences into a suitable algebraic analogue, and the image is an arithmetic group. Sullivan's proof of their result is incomplete as given. Triantafillou \cite{Triantafillou1999} points out that Sullivan's obstruction theory fails to account for a basepoint, and in turn their obstruction theory does not account for the basepoint of the rationalization. To ensure that our results have a correct foundation, in \cref{Appendix} we give a detailed proof of Sullivan's result. In \cref{Appendix} we also address the incongruence between Sullivan's \cite{Sullivan1977} notion of commensurability and the usual one.  

\begin{acknowledgements}
The author thanks Oscar Randal-Williams for an enlightening conversation on interpreting Gersten's refined torsion via $A$-theory (\cref{GerstenTorsionDefinition}). Thanks are also due to Fadi Mezher for pointing out the work of Roitberg \cite{Roitberg1984} and for verifying that \cite[Theorem E]{Mezher2024} applies to nilpotent spaces, and to Mark Powell for comments on an early draft. Finally, the author is grateful to their supervisor, Alexander Kupers, for their mentorship and conversations throughout the course of this project.
\end{acknowledgements} 

\section{Simple Homotopy Theory}

\subsection{Preliminaries}

Although the results in the introduction are stated for topological spaces, we will work with the $\infty$-category of spaces $\Spc$. As is common when transitioning to $\infty$-categories, some structures become properties. The structure of a topological space being finitely dominated, corresponds to the space being a compact object in $\Spc$. For an $\infty$-category $\mathcal C$ we write $\mathcal C^\omega$ for the full subcategory of compact objects. So, $\Spc^\omega $ is the category of spaces which can be finitely dominated. A finite space is a space equivalent to a finite colimit of the terminal object $*$. Although finite CW-complexes clearly correspond to finite spaces, we will see that the notion of a finite space is not enough to do simple homotopy theory, for this we will need a finite space with a finiteness structure (a notion which we introduce in this section). For our purposes, this gives the information we need to translate results about topological spaces into results about spaces, and thus we need not discuss traditional topological spaces again.

We must fix some notation in addition to what is already detailed above. $\Sp$ will be the $\infty$-category of spectra. If $X$ is a space, we will write $\mathcal C^X$ for the category of functors $\Fun(X, \mathcal C)$ where we regard $X$ as an $\infty$-groupoid. We call this the category of ($\mathcal C$-valued) local systems on $X$. With  few exceptions $\mathcal C$ will be $\Sp$. For a map of spaces $f \colon Y \to X$ we write $f^*$ for the pullback $\mathcal C^X \to \mathcal C^Y$. When $\mathcal C$ has colimits we write $f_!$ for the left adjoint to $f^*$. We remark that if $\mathcal C$ also has limits then $f_!$ preserves compact objects because it has a cocontinuous right adjoint.

The $A$-theory of a space $X$, $A(X)$, is the $K$-theory of local systems, $K((\Sp^X)^\omega)$. $A$-theory is given functoriality using $f_!$, as defined above. The assembly map is the natural transformation of functors $$\alpha_{-} \colon A(*) \otimes \Sigma^\infty_+ - \to A(-)$$ to the $A$-theory functor from its excisive approximation. The cofibre of $\alpha_X$ is the Whitehead spectrum $\Wh(X)$. We will write $A_n(X)$ or $\Wh_n(X)$ for the $n^{\mathrm{th}}$ homotopy group of the respective spectra. For more background on our approach to $A$-theory, see \cite{Lurie2014}.

If $X$ is a space then $\hAut(X)$ will be the $\infty$-group of self-equivalences of $X$. If $X$ is based, then the subscripts $*$ and $\pi_1$ will denote that those equivalences that fix the basepoint and the fundamental group respectively. When $X$ has a finiteness structure (\cref{FinitenessStr}), adding a superscript $s$ will denote those self-equivalences which are simple.

\subsection{Whitehead Torsion}

To begin, we recall the following $A$-theoretic definition of Whitehead torsion (see \cite{Lurie2014}). 

\begin{definition}\label{FinitenessStr}
	Suppose that $X$ is an object in $\Spc^\omega$. A \emph{finiteness structure} $\varphi_X$ for $X$ is a lift of the $K$-theory class of the constant local system $[\S_X]$ along the assembly map. That is, a finiteness structure is an explicit choice of equivalence $\alpha_X(\varphi_X) \simeq [\S_X]$. 
\end{definition}

For a finiteness structure to exist it must be true that $X$ is a finite space. The image of $[\S_X]$ in the cofibre of assembly $\Wh(X)$ is Wall's finiteness obstruction and so we can informally say that a finiteness structure is a guarantee that Wall's obstruction vanishes.

\begin{definition}
	Suppose that $f \colon Y \to X$ is an equivalence of spaces with finiteness structures denoted $\varphi_Y$ and $\varphi_X$ respectively. Then let $\varepsilon $ be the counit for the adjunction $f_! \dashv f^*$. Then $$f_! \S_Y \simeq f_! f^* \S_X \xrightarrow{\varepsilon} \S_X$$ gives a path in $A(X)$ with lifts of the ends along $\alpha_X$ (given by the finiteness structures). In other words we get a loop in $\Wh(X)$. The element $\tau(f) \in \Wh_1(X)$ represented by this loop is the \emph{Whitehead torsion} of $f$. If $\tau(f)$ vanishes we say that $f$ is \emph{simple}.
\end{definition}

To see how this relates to the usual definition of Whitehead torsion see \cite[Lec. 27, Prop 5]{Lurie2014}.

\subsection{Gersten Torsion}

We can now define Gersten's refinement of the Whitehead torsion \cite{Gersten1967}. In their original paper, Gersten claims that their construction works for any self-equivalence $f \colon X \to X$ of a compact space, but they only define it for $X$ a based connected space and for $f$ that acts as the identity on $\pi_1(X,x)$. As Ranicki \cite[\S 5]{Ranicki1987} points out, Gersten's definition cannot be directly generalized to all self-equivalences of compact spaces. Our definition works for any self-equivalence but it does not always refine the Whitehead torsion.

\begin{definition}\label{GerstenTorsionDefinition}
	Suppose that $X$ is a compact space. Let $\tau \colon X \to X_1$ be the map from $X$ to its $1$-truncation. Suppose that $f \colon X \to X$ is a self-equivalence and let $f_1 \colon X_1 \to X_1$ be the $1$-truncation of $f$. Then we can construct a self-equivalence $$\tau_! \S_X \xleftarrow{\varepsilon_{f_1}} {f_1}_! \tau_! \S_X \simeq \tau_! f_! \S_X \xrightarrow{\tau_! \varepsilon_f} \tau_! \S_X.$$ This gives a loop in $A(X_1)$. The element $\tau_G(f) \in A_1(X_1) \oset[-0.075cm]{\hspace{0.05cm}\cong}{\xleftarrow{\hspace{0.3cm}}} A_1(X)$ represented by this loop is the \emph{Gersten torsion} of $f$. 
\end{definition}

We momentarily delay the justification of the isomorphism  $A_1(X_1) \oset[-0.075cm]{\hspace{0.05cm}\cong}{\xleftarrow{\hspace{0.3cm}}} A_1(X)$ to make the following remark.

\begin{remark}
	In the case where $X$ is a connected based compact space and $f$ is based and induces the identity on $\pi_1(X,x)$, \cref{GerstenTorsionDefinition} can be simplified because then there is a contractible choice of based equivalence $f_! \simeq \id_{X_1}$. Indeed, since $\Hom_{\Spc^{/*}}(X_1,X_1)$ is an $E_1$-space, all of its path components are the same. Then $\pi_k(\Hom_{\Spc^{/*}}(X_1,X_1), *)$ vanishes using the tensor--hom adjunction. Thus the Gersten torsion is the element of $A_1(X_1)$ is associated to $$\tau_!\S_X \simeq \tau_! f_! \S_X \xrightarrow{\tau_! \varepsilon_f} \tau_!\S_X.$$ This is the scenario in which the Gersten torsion is a lift of the Whitehead torsion, and so going forward we work only in this simplified case. Our original definition is useful as it makes it clear that $\tau_G$ depends on $f$ only up to equivalence (rather than based equivalence).   
\end{remark}

\begin{proposition}\label{ConnectivityProposition}
	For a space $X$, the map $\tau \colon X \to X_1$ induces an isomorphism on $A_1$ and $\Wh_1$. 
\end{proposition}

\begin{proof}
If $X$ is connected, we have that the Schwede--Shipley theorem (originally proved in \cite{Schwede2003}, but for our context \cite[Theorem 7.1.2.1]{Lurie2017}) gives an equivalence between local systems of spectra on $X$ and modules of the spherical group ring $\Sigma^{\infty}_+ \Omega X$ $$\Sp^X \simeq  \mathrm{Mod}_{\Sigma^{\infty}_+ \Omega X}$$ which means $A(X) \simeq K(\Sigma^{\infty}_+ \Omega X)$. Since the map $\tau$ was $2$-connected the map on spherical group rings is $1$-connected and thus the map $$K(\Sigma^{\infty}_+ \Omega X) \to K(\Sigma^{\infty}_+ \Omega X_1)$$
is again $2$-connected \cite[Proposition 1.1]{Waldhausen1978}. The $A$-theory of a disjoint union is the sum of $A$-theory spectra for each component, thus for any $X$ (possible not connected) the map $A(X) \to A(X_1)$ is $2$-connected. This gives that the map on $A_1$ is an equivalence.

Clearly the map $$\tau \colon A(*) \otimes \Sigma_+^\infty X \to A(*) \otimes \Sigma_+^\infty X_1$$ is $2$-connected and so the map on the cofibre of assembly $ \Wh (X) \to \Wh(X_1)$ is $2$-connected. This gives that the map on $\Wh_1$ is an equivalence. 
\end{proof}

\begin{remark}
We take a moment to outline how our \cref{GerstenTorsionDefinition} compares with Gersten's original refinement of the Whitehead torsion \cite{Gersten1967}. We never use Gersten's original formulation.

Suppose that $X$ is a connected based compact space and that $f \colon X \to X$ is a based equivalence inducing the identity on $\pi_1(X,x)$. Gersten's element of $K_1$ is the one associated to the map of $\Z[\pi_1(X,x)]$-chain complexes $$\tilde f \colon C(\tilde X; \Z) \to C(\tilde X; \Z)$$
induced by $f$ on the chains of the universal cover $\tilde X$ of $X$. In the classical definition of the Whitehead torsion of $Y \to X$, the map is used to view $C(\tilde Y,\Z)$ as a $\Z[\pi_1(X,x)]$-chain complex. For the map $\tilde f$ to be a map of $\Z[\pi_1(X,x)]$-chain complexes, without first using $f$ to push forward the domain chain complex, is where Gersten requires the assumption that $f$ acts as the identity on $\pi_1(X,x)$. 

Now in order to obtain Gersten's definition from ours we can push forward from local systems of spectra to a local systems of modules over $H\Z$  $$(\Sp^{X_1})^\omega \to (\mathrm{Mod}^{X_1}_{H\Z})^\omega.$$
This is equivalent, again via the Schwede--Shipley theorem \cite[Theorem 7.1.2.1]{Lurie2017}, to the map on modules given by the linearization map on $E_1$-rings
$$\Sigma^\infty_+ \Omega X_1 \to H\Z\otimes \Sigma^\infty_+ \Omega X_1 \simeq H\Z[\pi_1(X,x)].$$
This map of $E_1$-rings is $1$-connected and so the induced map on $K$-theory is $2$-connected \cite[Proposition 1.1]{Waldhausen1978}. Under this equivalence $\tau_! \S_X \otimes H\Z$ is equivalent to $C(\tilde X, \Z)$ and the map $\tau_! \varepsilon$ is exactly the map $\tilde f$. 
\end{remark}

Finally we show that our Gersten torsion is a lift of the Whitehead torsion.

\begin{proposition}\label{GerstenWhiteheadComparison}
	Suppose that $X$ is a connected based finite space. Suppose that $X$ is given some finiteness structure $\varphi_X$. Let $f \colon X \to X$ be an equivalence that acts as the identity on $\pi_1(X,x)$. Then the image of $\tau_G(f)$ in $\Wh_1(X_1) \oset[-0.075cm]{\hspace{0.05cm}\cong}{\xleftarrow{\hspace{0.3cm}}} \Wh_1(X)$ is $\tau(f)$.
\end{proposition}

\begin{proof}
	Pushing forward the finiteness structure $\varphi_X$ for $X$ along $\tau$ gives a lift $\tau \varphi_X$ of $\tau_! [\S_X]$ along the assembly map $\alpha_{X_1}$. This gives us lifts along assembly for the endpoints of the path in $A(X_1)$ defined by the equivalence 
	$$ \tau_!f_! \S_X \xrightarrow{\tau_! \varepsilon} \tau_! \S_X.$$
	The resulting loop in $\Wh_1(X)$ is exactly the Whitehead torsion $\tau(f)$ under the equivalence $\Wh_1(X_1) \oset[-0.075cm]{\hspace{0.05cm}\cong}{\xleftarrow{\hspace{0.3cm}}} \Wh_1(X)$ induced by $\tau$ (\cref{ConnectivityProposition}). Now consider the equivalence $\tau f \simeq \tau$. This gives us an equivalence of lifts between $\tau f \varphi_X$ lifting $\tau_! f_! [\S_X]$ and $\tau \varphi_X$ lifting $\tau_! [\S_X]$. So the Whitehead torsion is also given by the loop resulting from using $\tau \varphi_X$ as a lift of the endpoints of the path in $A(X_1)$ defined by $$\tau_! \S_X \xrightarrow{\tau_! \varepsilon} \tau_! \S_X.$$	
	But now this loop is just the image of the Gersten torsion $\tau_G(f)$ moved from the basepoint $[\S_X]$ to the basepoint $\alpha_{X_1}(\tau \varphi_X)$ using the equivalence $\alpha_{X_1}(\tau \varphi_X) \simeq [\S_X]$.  
\end{proof}

\section{Main Results}

\subsection{Vanishing Theorem}

Our vanishing theorem will use the following swindle, which is a consequence of the additivity theorem for $K$-theory. 

\begin{proposition}\label{Swindle}
	Suppose that $A$ is an object in some pointed $\infty$-category $\mathcal C$ with finite colimits. Suppose that $f$ and $g$ are two commuting self-equivalences of an object $A$ in $\mathcal C$. Let $h$ be the self-equivalence induced by $g$ on $\mathrm{cofib} f$. The element $[h]$ of $K_1 (\mathcal C)$ given by $h$ is $0$.
\end{proposition}

The assumptions on $\mathcal C$ are to ensure that $K(\mathcal C)$ can be defined using the $S_\bullet$-construction (originally due to \cite{Waldhausen1985} but in this context see \cite[Lec. 16]{Lurie2014}).

\begin{proof}
Choosing an equivalence $f g \simeq gf$ gives a self-equivalence of the object $(g \colon A \to A)$ in $\Fun(\Delta^1, \mathcal C)$ and thus determines a loop in $K(\Fun(\Delta^1, \mathcal C))$. The additivity theorem for $K$-theory \cite[Lec. 17, Theorem 1]{Lurie2014} says that the map $K(\Fun(\Delta^1, \mathcal C)) \to K(\mathcal C)$ given by taking the target is the sum of the maps given by taking the source and cofibre (for the original proof and the comparison of different versions of the additivity theorem, see \cite[\S 1.3-1.4]{Waldhausen1985}). Thus we immediately obtain $[g] = [g] + [h]$ meaning that $[h] =0$.
\end{proof}

The proof of our vanishing result will proceed by fibring over the circle, a long tradition for vanishing results in simple homotopy theory initiated by Mather \cite{Mather1965}.

\begin{proposition}\label{VanishingProposition}
	Suppose that we have a based map of connected spaces $p \colon X \to S^1$ which is surjective on the fundamental group. Suppose that the fibre $Y$ and $X$ are compact. Then the torsion of any based self-equivalence $f \colon X \to X$ that acts as the identity on $\pi_1(X,x)$ vanishes.
\end{proposition}

\begin{proof}
	The map $f$ commutes with the projection to $S^1$ and so we get an induced map on fibres $g \colon Y \to Y$. Let $i$ be the inclusion of the fibre $Y \to X$ and let $m \colon Y \to Y$ be the monodromy map given by going around the circle. We have that $X$ is built as the coequalizer
\[\begin{tikzcd}[cramped]
	Y & Y
	\arrow["m", shift left, from=1-1, to=1-2]
	\arrow["\id"', shift right, from=1-1, to=1-2]
\end{tikzcd}.\]Now we have a homotopy $im \simeq i$ given by going around $S^1$. This means we can view this coequalizer diagram in the category $\Spc_{/X}$ of spaces over $X$, where the objects are $i \colon Y \to X$. Then the coequalizer is exactly $\id \colon X \to X$. Now before the map $g$ was a map $Y \to Y$, but if we view $Y$ as a space over $X$ then $g$ is a map from $ig \colon Y \to X$ to $i \colon Y \to X$ because the diagram
\[\begin{tikzcd}
	Y && Y \\
	& X
	\arrow["g", from=1-1, to=1-3]
	\arrow["ig"', from=1-1, to=2-2]
	\arrow["i", from=1-3, to=2-2]
\end{tikzcd}\]
commutes. In this way, $g$ gives a map of coequalizer diagrams in $\Spc_{/X}$ because $g$ commutes with $m$ and $\id$. On the domain of this map of coequalizer diagrams are the spaces over $X$ given by $ig \colon Y \to X$ and on the codomain are the spaces over $X$ given by $i \colon Y \to X$. A map of coequalizer diagrams over $X$ requires us to provide $3$-morphism, and in this case one is guaranteed by horn filling. The map induced on the coequalizers is then $f$ from $f \colon X \to X$ to $\id \colon X \to X$. 

Since $X$ and $Y$ are compact this map of coequalizer diagrams is in the category $\Spc_{/X}^\omega$. We can pass the entire map of coequalizer diagrams through the composition $$\Spc_{/X}^\omega \to (\Spc^X)^\omega \to (\Sp^X)^\omega$$
where the first functor is Lurie's straightening functor \cite[Theorem 2.2.1.2]{Lurie2009} and the second is $\Sigma^\infty_+$ applied to the target. This composition sends an object, for example $i \colon Y \to X$, to $i_!i^* \S_Y \simeq i_! \S_Y$ and sends a morphism, for example $g$ from $ig \colon Y \to X$, to $i \colon Y \to X$ to $i_!\varepsilon_g$ where $\varepsilon_g$ is the counit for the adjunction $g_! \dashv g^*$. This composition also preserves colimits so our coequalizers remain except, because we are now in a stable category, they become cofibre of differences. Thus, we obtain the diagram
\[\begin{tikzcd}[column sep=4em]
	{i_!g_!\S_Y} & {i_!g_!\S_Y} & {f_!\S_X} \\
	{i_!\S_Y} & {i_!\S_Y} & {\S_X}
	\arrow["{i_!g_!(\varepsilon_m - \varepsilon_\id)}", from=1-1, to=1-2]
	\arrow["{i_! \varepsilon_g}"', from=1-1, to=2-1]
	\arrow[from=1-2, to=1-3]
	\arrow["{i_! \varepsilon_g}"', from=1-2, to=2-2]
	\arrow["{\varepsilon_f}"', from=1-3, to=2-3]
	\arrow["{i_!(\varepsilon_m - \varepsilon_\id)}"', from=2-1, to=2-2]
	\arrow[from=2-2, to=2-3]
\end{tikzcd}\]
where the horizontal sequences are exact. Let $\tau$ be the truncation map from $X \to X_1$. We now apply $\tau_!$ to this diagram. $\tau _! f_! \simeq \tau_!$ because $f$ induces the identity on the fundamental group. Similarly, $\tau_! i_! g_! \simeq i_! \tau_! g_! \simeq i_! \tau_! \simeq \tau_! i_!$ (where we have abused notation and also allowed $\tau$ to denote the map $Y \to Y_1$). This gives the commutative diagram 
\[\begin{tikzcd}[column sep=4em]
	{\tau_!i_!\S_Y} & {\tau_!i_!\S_Y} & {\tau_!f_!\S_X} \\
	{\tau_!i_!\S_Y} & {\tau_!i_!\S_Y} & {\tau_!\S_X}
	\arrow["{\tau_! i_!(\varepsilon_m - \varepsilon_\id)}", from=1-1, to=1-2]
	\arrow["{\tau_! i_! \varepsilon_g}"', from=1-1, to=2-1]
	\arrow[from=1-2, to=1-3]
	\arrow["{\tau_! i_! \varepsilon_g}"', from=1-2, to=2-2]
	\arrow["{\tau_! \varepsilon_f}"', from=1-3, to=2-3]
	\arrow["{\tau_! i_!(\varepsilon_m - \varepsilon_\id)}"', from=2-1, to=2-2]
	\arrow[from=2-2, to=2-3]
\end{tikzcd}\]
in $(\Sp^{X_1})^\omega$. Applying $K$-theory, the right-hand map becomes the loop which gives the $\tau_G(f)$, and applying \cref{Swindle} to the left-hand square this vanishes. 
\end{proof}

This proposition is the main piece of input we need for \cref{MainVanishingThm}. We will also need a lemma. 

\begin{lemma}\label{NilpotentGroupSurjectivity}
	If $G$ is an infinite finitely generated nilpotent group then there is a surjective map $G \to \Z$.
\end{lemma}

\begin{proof}
	Since $G$ is finitely generated nilpotent, it has a finite normal torsion subgroup $T$ \cite[Corollary 1.A.10]{Segal1983}. Now we have a surjective map from $G \to G/T$. $G/T$ now satisfies all the assumptions of the original group and it is also torsion free, so we can assume that our original group $G$ was torsion free.
	
	We know that $G$ embeds into the group of $n \times n$ unitriangular matrices with integer coefficients \cite{Malcev1949}. Let $\mathrm{Uni}_n^m(\Z)$ be the subgroup of unitriangular matrices where all entries in the $1^\text{st}$ to $m^\text{th}$ off diagonals are zero. So $m$ is between $0$ and $n-1$ with $\mathrm{Uni}_n^0(\Z)$ being all unitriangular matrices and $\mathrm{Uni}_n^n(\Z)$ is the trivial group. This defines a filtration on all unitriangular matrices with $m^\text{th}$ graded piece $\Z^{n - m}$. Because $G$ is non-zero, it has non-trivial image in one of these graded pieces. The image is always a product of infinite cyclic groups, so in particular $G$ surjects to $\Z$.
\end{proof}

\begin{theorem}[\cref{MainVanishingThm}]\label{VanishingTheorem}
	Let $X$ be a connected based finite nilpotent space with infinite fundamental group. Then the Whitehead torsion of any based self-equivalence $f \colon X \to X$ that acts as the identity on $\pi_1(X,x)$ vanishes.
\end{theorem}

\begin{proof}
	Using \cref{NilpotentGroupSurjectivity}, there is a map from $\pi_1(X)$ to $\Z$. Then we have the map $X \to B\pi_1(X) \to B \Z \simeq S^1$. Then, by work of Mislin \cite[Theorem A, Lemma 1.1]{Mislin1976}, the fibre of this map is always a compact object. Thus by \cref{VanishingProposition}, $\tau_G(f) = 0$ and so by \cref{GerstenWhiteheadComparison} $\tau(f)$ vanishes.
\end{proof}

\begin{remark}
	The assumption in \cref{VanishingTheorem} that $f$ acts the identity on the fundamental group of $X$ is indeed required. Ferry \cite{Ferry1981} defines an injective map
	\begin{align*}
		\widetilde K_0(G) &\to \Wh(\Z \times G)\\
		w(X) &\mapsto \tau(S^1 \times X \xrightarrow{(-1, \id_X)} S^1 \times X)
	\end{align*}
	where $w(X)$ is the Wall finiteness obstruction of $X$, and to make sense of the torsion on $S^1 \times X$ some finiteness structure is placed on $S^1 \times X$.  Mislin \cite[Theorem 4.1]{Mislin1975} shows there exist compact nilpotent spaces $X$ such that $w(X)$ does not vanish and thus we can construct a self-equivalence of $S^1 \times X$ which is not simple.
\end{remark}

\subsection{Moduli Theorem}\label{MainCorollariesSection}

We can now show that the moduli space of simple structures can often have infinitely many components. By a simple structure on a space $X$ we mean a finiteness structure modulo equivalence by simple equivalences.

\begin{corollary}[\cref{MainStructureThm}]\label{StructureCorollary}
	Let $X$ be a connected finite nilpotent space. Suppose that $\Wh_1(X)$ is infinite and that the fundamental group of $X$ is infinite with only finitely many automorphisms. Then there are infinitely many distinct simple structures on $X$.
\end{corollary}

\begin{proof}
	Pick any basepoint for $X$. Cockcroft--Moss \cite[Theorem 1]{Cockcroft1972} show that simple structures on $X$ are in bijection with the orbits of $\Wh_1(X)$ under a certain right action of $\pi_0\hAut_*(X)$. Give $X$ some finiteness structure. Then suppose that $\psi \colon Y \to X$ is an equivalence between spaces where $Y$ has a finiteness structure. Then $f \in \pi_0\hAut_*(X)$ acts by sending $\tau(\psi)$ to $$f_!^{-1} (\tau(\psi)) + \tau(f).$$
	In our scenario, we can see that this action factors through the map $\pi_0\hAut_*(X) \to \Aut(\pi_1(X,x))$. This is obvious for the term $f_!^{-1} (\tau(\psi))$. For the term $\tau(f)$, the map given by the Whitehead torsion $$\tau \colon \pi_0 \hAut_*(X) \to \Wh_1(X)$$ vanishes on the kernel of the map $\pi_0\hAut_*(X) \to \Aut(\pi_1(X,x))$ by \cref{VanishingTheorem} and thus factors over $\Aut(\pi_1(X,x))$.
	
	So, the simple structures on $X$ are in bijection with the orbits of the infinite set $\Wh_1(X)$ under the action of the finite group $\Aut(\pi_1(X,x))$. This shows that there are infinitely many simple structures on $X$.
\end{proof}

We will now present examples of spaces which satisfy the assumptions of \cref{StructureCorollary}. We will need to verify that $\Wh_1$ is infinite and for this we will use the Bass--Heller--Swan decomposition \cite[Chapter VII]{Bass1968}
$$\Wh(\Z \times G) \iso \Wh(G) \times \tilde K_0(\Z[G]) \times \widetilde \Nil_1(\Z[G]) \times \widetilde \Nil_1(\Z[G]).$$
Here $\Wh(G)$ is the Whitehead group of a group $G$, equivalent to $\Wh_1(X_1)$ in our notation. An important group satisfying the conditions on the fundamental group in \cref{StructureCorollary} is $\Z \times \Z/n$. We can classify exactly when $\Wh(\Z \times \Z/n)$ is infinite.

\begin{proposition}\label{WhiteheadGroupCalc}
	The group $\Wh(\Z \times \Z/n)$ is infinite if and only if $n \neq 1,2,3,6$	
\end{proposition}
 
\begin{proof}
	The groups $\widetilde \Nil_1(\Z/n)$ are infinite if and only if they are non-zero if and only if $n$ is not square free (see \cite[Theorem 5.8]{Nagy2023} for a fully assembled proof). The Whitehead group $\Wh(\Z/n)$ is infinite if and only if $n \neq 1,2,3,4,6$ \cite[Corollary 6.5]{Milnor1966}. Finally $\tilde K_0(\Z[\Z/n])$ is always finite \cite[Proposition 9.1]{Swan1960}. Assembling these using the Bass--Heller--Swan decomposition gives the desired result.
\end{proof}

The following spaces satisfy the assumptions of \cref{StructureCorollary}.

\begin{example}\label{ex1}
	The lens space $S^1 \times L(p,q_1,\dots, q_n)$ for $p \neq 2,3$. This has infinite Whitehead group by \cref{WhiteheadGroupCalc}. This space has abelian fundamental group and the deck transformations on the universal cover are all the identity, so it is nilpotent.
\end{example}

\begin{example}\label{ex2}
	$S^1 \times S^3 / Q_8$. This has infinite Whitehead group because $\widetilde \Nil_1(Q_8)$ is infinite \cite[Proposition 52]{Guaschi2018}. It is easy to see that $Q_8$ is nilpotent, and again the deck transformations are all the identity so it is a nilpotent space. 
\end{example}

\begin{example}\label{ex3}
	The Lie group $U(n)/\langle \zeta \rangle$ for an $n^\text{th}$ root of unity $\zeta$ and with $n \neq 1,2,3,6$. This has infinite Whitehead group by \cref{WhiteheadGroupCalc}. This space is nilpotent because it is a Lie group.
\end{example}

\begin{remark}
	In many cases one can combine our vanishing result with the techniques of Nagy--Nicholson--Powell \cite{Nagy2023} to prove a stronger result than \cref{StructureCorollary}. We can find many infinite families of h-cobordant manifolds none of which are simply homotopy equivalent. If $M$ is a compact connected $d$-manifold with some some suitable assumptions to ensure that surgery is possible, \cite[Theorem 4.16]{Nagy2023} characterizes the collection of manifolds h-cobordant to $M$ modulo simple homotopy equivalence. Let $\pi_0\hAut(M)$ act on $\Wh_1(M)$ by having $g \colon M \to M$ take $\tau(f \colon N \to M)$ to $\tau(g \circ f)$. Let $$q \colon \Wh_1(M) \to \Wh_1(M)/ \pi_0\hAut(M)$$ be the quotient map. Let $\omega$ be the orientation character of $M$ and define $$\mathcal I_d(\pi_1(M), \omega ) = \{x - (-1)^d \bar x\  |\ x \in \Wh(\pi_1(M))\}$$ where $x \mapsto \bar x$ is some involution depending on $\omega$. Then the collection of manifolds h-cobordant but not simple homotopy to $M$ is equivalent to $q(I_d(\pi_1(M), \omega ))$. In the case where $M$ is nilpotent with infinite fundamental group, \cref{VanishingTheorem} gives that the action of $\pi_0\Aut(M)$ factors through $\Out(\pi_1(M))$. If $\Out(\pi_1(M))$ is finite, then one simply needs to verify that $\mathcal I_d(\pi_1(M), \omega )$ is infinite to find our infinite family. Notably, when $\pi_1(M) \iso \Z \times G$, \cite[Proposition 5.10]{Nagy2023} shows that $\widetilde{\Nil}_1(\Z[G])$ is a subset of $\mathcal I_d(\pi_1(M), \omega )$. Thus for $n$ not square free when relevant, we see that the Examples \ref{ex1}, \ref{ex2}, and \ref{ex3} give infinite families of h-cobordant manifolds which are not simple homotopy equivalent.  
\end{remark}

\subsection{Arithmeticity Theorem}

We conclude by using our vanishing result to prove \cref{MainArithmeticityThm}.

\begin{theorem}[\cref{MainArithmeticityThm}]\label{ArithmetictyCorollary}
	Let $X$ be a connected based nilpotent space with a finiteness structure. The group $\pi_0 \hAut_{*, \pi_1}^s(X)$ is commensurable to an arithmetic group.
\end{theorem}

\begin{proof}
The group $\pi_0 \hAut_*(X)$ is residually finite by work of Roitberg \cite[Theorem 3]{Roitberg1984} and so all its subgroups are residually finite. Thus, to show that $\pi_0 \hAut^s_{*,\pi_1}(X)$ is commensurable to an arithmetic group we just need to show that there is a surjection with finite kernel from $\pi_0 \hAut^s_{*,\pi_1}(X)$ to an arithmetic group. That is, we need to show that $\pi_0 \hAut^s_{*,\pi_1}(X)$ is commensurable up to finite kernel to an arithmetic group. 

If $X$ has finite fundamental group this is just \cite[Theorem 2.2]{Bustamante2023}. 

Now assume that $X$ has infinite fundamental group. In this case \cref{VanishingTheorem} immediately gives that the Whitehead torsion of any element of $\pi_0 \hAut_{*,\pi_1}(X)$ vanishes, so $\pi_0 \hAut^s_{*,\pi_1}(X) = \pi_0 \hAut_{*,\pi_1}(X)$. So we just need to use rational homotopy theory to show that $\pi_0 \hAut_{*,\pi_1}(X)$ is commensurable up to finite kernel to an arithmetic group.

To begin, we have a commutative square of groups
\[\begin{tikzcd}
	{\pi_0 \hAut_*(X)} & {\Aut(\pi_1(X))} \\
	{\pi_0 \hAut_*(X_\Q)} & {\pi_0 \hAut_*((X_1)_\Q)}
	\arrow[from=1-1, to=1-2]
	\arrow[from=1-1, to=2-1]
	\arrow[from=1-2, to=2-2]
	\arrow[from=2-1, to=2-2]
\end{tikzcd}.\]
Now let $\mathcal M$ be the minimal model for $X$, in the sense of Sullivan \cite{Sullivan1977}. So $\mathcal M$ is a commutative differential graded algebra (CDGA) with augmentation coming from the basepoint of $X$. The bottom map in the square above is the map $$\Aut(\mathcal M) / {\sim_*} \to \Aut(\mathcal M^1) / {\sim_*}$$
given by restricting to the model generated by generators in degree at most $1$. Here $\sim_*$ is the equivalence relation of augmentation preserving CDGA homotopy. This map is algebraic because we are taking a map of algebraic groups and quotienting by an algebraic subgroup, see the proof of \cref{AlgebraicityCorollary} for details. Now if we take the kernels of the horizontal maps in the above square we get a map $$\pi_0\hAut_{*,\pi_1}(X) \to \Aut_{\mathcal M ^ 1}(\mathcal M) / {\sim_*}$$
where $\Aut_{\mathcal M ^ 1}(\mathcal M) / {\sim_*}$ is an algebraic subgroup of $\Aut(\mathcal M) / {\sim_*}$. Now, \cref{FiniteKernelImageArithmetic} says that the vertical homomorphisms of our square have finite kernel and image an arithmetic group. An application of \cref{KernelAlgGroups} gives that the map $\pi_0\hAut_{*,\pi_1}(X) \to \Aut_{\mathcal M ^ 1}(\mathcal M) / {\sim_*}$ has finite kernel and image an arithmetic group.
\end{proof}

There are various extensions of this result one might hope to hold.

\begin{question}
	With the assumptions in \cref{ArithmetictyCorollary}, is $\pi_0 \hAut^s_{*}(X)$ (here we are not fixing $\pi_1(X,x)$) commensurable to an arithmetic group? Is $\pi_0 \hAut^s_{*}(X)$ of finite type? 
\end{question}

\begin{question}\label{PseudoisotopyQuestion}
	If we assume that $M$ is a nilpotent manifold, is the group of pseudoisotopy classes of diffeomorphisms (possibly restricting to those that act as the identity on the fundamental group of $M$) commensurable to an arithmetic group?
\end{question}

One cannot hope for the group of isotopy classes of diffeomorphisms of $M$ to be commensurable to an arithmetic group because Hsiang--Sharpe \cite[p. 408, Example 1]{Hsiang1976} show it may not even be finitely generated. This is a problem which cannot be solved by fixing the fundamental group, the group will still be infinitely generated.

To answer \cref{PseudoisotopyQuestion} one might hope to use the techniques of Bustamante--Krannich--Kupers \cite[Theorem 2.6]{Bustamante2023}, but this is not straightforward because the $L$-theory groups they consider may be infinitely generated in this case. 

\appendix

\section{Arithmeticity of Self-equivalences}\label{Appendix}

The purpose of this appendix is to provide a detailed proof of the following theorem, which is used to prove \cref{ArithmetictyCorollary} and also the foundation for the proofs of other finiteness results about various automorphism groups (e.g. \cite{Kupers2025, Bustamante2023}).

\begin{theorem}[Sullivan {\cite[Theorem 6.1]{Sullivan1977}}]\label{ArithmeticityHomotopyAutomorphisms}
	If $X$ is a connected based finite nilpotent space then $\pi_0 \hAut_*(X)$ is commensurable to an arithmetic group.
\end{theorem}

The main issue with Sullivan's proof, originally observed by Triantafillou \cite[p. 284]{Triantafillou1999}, is that Sullivan only deals with unbased spaces, but the obstruction theory they cite requires a basepoint. In this sense, our \cref{ArithmeticityHomotopyAutomorphisms} is not actually equivalent to Sullivan's theorem \cite[Theorem 6.1]{Sullivan1977}. At the end of this appendix we obtain Sullivan's original theorem statement as \cref{SullivanOriginalTheorem}. Although Triantafillou notices that Sullivan's obstruction theory requires a basepoint, they fail to account for the corresponding basepoint in their rational homotopy theory. Accounting for this basepoint requires some care. Triantafillou also fills in many details which Sullivan's proof omits, but some sections of their proof are still not completely precise. We hope to give a more precise account.

When Sullivan \cite{Sullivan1977} says that two groups are ``commensurable'' they means commensurable up to finite kernel. The solution to the incongruence of Sullivan's terminology with the literature was pointed out by Serre, \cite[1.2 (8)]{Serre1979} who observed that one could use Sullivan's profinite homotopy theory \cite[Theorem 3.2]{Sullivan1974} to show that the group $\pi_0 \hAut(X)$ is residually finite. This idea was carried out by Mezher \cite[Theorem E]{Mezher2024}. If a group is residually finite, then being commensurable up to finite kernel implies being commensurable. We do not use Sullivan's profinite homotopy theory to prove \cref{ArithmeticityHomotopyAutomorphisms} because we are working with based homotopy self-equivalences, instead we use work of Roitberg \cite[Theorem 1.3]{Roitberg1984}.

There are other similar results in the literature. Before Sullivan, Wilkerson \cite{Wilkerson1976} gives a proof for simply connected spaces (in the simply connected case, there are no issues with Sullivan's proof either). Dror--Dwyer--Kan \cite{Dror1981} give a proof that for virtually nilpotent $X$, the group $\pi_0 \hAut(X)$ is of finite type. Although we are primarily interested in arithmetic groups because they are of finite type, in practice being arithmetic has useful consequences. In particular, although finite type groups are closed under extensions, they are almost never closed under taking subgroups. On the other hand, if we know that a subgroup of an arithmetic group is given by intersecting with the kernel of an algebraic map, then this subgroup is again arithmetic. we use this explicitly in the proof of \cref{ArithmetictyCorollary}. 

\

To begin working towards Sullivan's theorem, we will state some facts about arithmetic groups. By algebraic group, we will always mean affine rational algebraic group. 

\begin{fact}[{\cite[p. 106]{Serre1979}}]\label{GeneralAlgebraicGroupFacts}
Let $\varphi\colon G \rightarrow H$ be a map of algebraic groups and let $\Gamma \leq G$ be an arithmetic subgroup. 
\begin{enumerate}
		\item If $\varphi$ is surjective then $\varphi(\Gamma)$ is an arithmetic subgroup.
		\item $\ker(\varphi)$ is an algebraic group and $\ker(\varphi) \cap \Gamma$ is an arithmetic subgroup.
\end{enumerate}
\end{fact}

\begin{fact}[{\cite[Corollary 3.8 (1)]{Chatterjee2015}}]\label{SpecialAlgebraicGroupFact}
	Let $U$ and $G$ be algebraic groups with $U$ unipotent. Then any group homomorphism $U \to G$ is a map of algebraic groups. 
\end{fact}

We will prove \cref{ArithmeticityHomotopyAutomorphisms} by induction up the Postnikov tower of $X$. For the induction step we will use the next two facts. The first result was originally attributed in \cite{Triantafillou1999} without reference to Borel and proved by Kupers. 

\begin{fact}[{\cite[Lemma 2.10]{Kupers2025}}]\label{ExactSequenceAlgGroups}
	Suppose we have a map of short exact sequences of groups
\[\begin{tikzcd}
	1 & A & {A'} & {A''} & 1 \\
	1 & G & {G'} & {G''} & 1
	\arrow[from=1-1, to=1-2]
	\arrow[from=1-2, to=1-3]
	\arrow["r", from=1-2, to=2-2]
	\arrow[from=1-3, to=1-4]
	\arrow["{r'}", from=1-3, to=2-3]
	\arrow[from=1-4, to=1-5]
	\arrow["{r''}", from=1-4, to=2-4]
	\arrow[from=2-1, to=2-2]
	\arrow[from=2-2, to=2-3]
	\arrow[from=2-3, to=2-4]
	\arrow[from=2-4, to=2-5]
\end{tikzcd}\]
	with the bottom a sequence of algebraic groups and with $G$ unipotent. Further suppose that $r$ and $r''$ both have finite kernel and image an arithmetic group. Then $r'$ also has finite kernel and image an arithmetic group. 
\end{fact}

We do not know of a reference for our final fact, so we provide a proof.

\begin{fact}\label{KernelAlgGroups}
	Suppose we have a map of exact sequences of pointed sets
\[\begin{tikzcd}[cramped]
	1 & A & {A'} & {A''} \\
	1 & G & {G'} & {G''}
	\arrow[from=1-1, to=1-2]
	\arrow["i", from=1-2, to=1-3]
	\arrow["r", from=1-2, to=2-2]
	\arrow["q", from=1-3, to=1-4]
	\arrow["{r'}", from=1-3, to=2-3]
	\arrow["{r''}", from=1-4, to=2-4]
	\arrow[from=2-1, to=2-2]
	\arrow["I"', from=2-2, to=2-3]
	\arrow["Q", from=2-3, to=2-4]
\end{tikzcd}\]
	where $i$, $r$ and $r'$ are maps of groups and $I$ is a map of algebraic groups. Further suppose that $r'$ and $r''$ both have finite kernel and $r'$ has image an arithmetic subgroup. Suppose that the map $q$ satisfies that $q(y) = q(y')$ implies $q(y'y^{-1}) = 1$. Then $r$ has finite kernel and image an arithmetic subgroup.
\end{fact}

\begin{proof}
	The map $r$ has finite kernel because $r'$ does. The group $\im(r') \cap \im(I)$ is an arithmetic subgroup of  $\im(I)$ \cite[p. 106]{Serre1979}. Now $\im(r)$ is a subgroup of $\im(r') \cap \im(I)$. Let $x, x' \in \im(r') \cap \im(I)$ represent different left-cosets of $\im(r)$ in $\im(r') \cap \im(I)$. Since $x, x' \in \im(I)$ we get that $Q(x) = 1 = Q(x')$. Let $y$ and $y'$ be lifts of $x$ and $x'$ under $r'$. Define $z = q(y)$ and $z' = q(y')$. Then, commutativity of the above diagram gives that $z, z' \in \ker (r'')$. If $z = z'$ then $q(y' y^{-1}) = 1$ i.e. $i(w)y = y'$ for some $w \in A$. Then $r(w) x = x'$ contradicting our assumption that $x$ and $x'$ represent different left-cosets. Thus, different left-cosets give distinct elements of $\ker(r'')$. This shows that $\im(r)$ is a finite index subgroup of $\im(r') \cap \im(I)$ and is therefore arithmetic.
\end{proof}

We can now begin the rational homotopy theory. To start, we will build the algebraic group in which the group of based self-equivalences is arithmetic. For the remainder of this appendix $X$ is a based, connected, nilpotent space or perhaps a Postnikov truncation thereof. Let $(\mathcal M, d)$ be the minimal model of $X$ in the sense of Sullivan \cite{Sullivan1977}. So, $(\mathcal M, d)$ is a commutative differential graded algebra (CDGA). $\mathcal M$ also has an augmentation coming from the basepoint of $X$. 

\begin{proposition}
	Let $\mathfrak u$ be the set of all elements of the form $[i,d] = di + id$ where $i$ is a degree $-1$ derivation of $\mathcal M$. Here $[-,-]$ is the Lie bracket on the graded Lie algebra of all graded derivations of $\mathcal M$. Further, define $\mathfrak u_*$ to be the subset of $\mathfrak u$ such that the derivation $i$ vanishes on the degree $1$ elements $\mathcal M$. Then, $\mathfrak u$ and $\mathfrak u_*$ are nilpotent Lie algebras.
\end{proposition}

\begin{proof}
	First $$[d, di + id] = ddi + did - did - idd = 0.$$
	Then using this $$[di + id, dj + jd] = d\underbrace{[i, jd + dj]}_{k} + \underbrace{[i, jd + dj]}_{k}d$$
	so $\mathfrak u$ is closed under commutators. Now suppose that $i$ and $j$ vanish on all degree $1$ elements and $x$ is such an element. Then $$k(x) = (ijd + idj - jdi - dji)(x) = ijd(x).$$
	The minimality of $\mathcal M$ gives that $d(x)$ is a linear combination of products of elements in degree $1$, so $ijd(x) =0$. This gives that $\mathfrak u_*$ is closed under commutators. Since $\mathfrak u$ and $\mathfrak u_*$ are also closed under taking sums we get they are Lie algebras.
			
	Define the weight of a generator $x$ of $\mathcal M$ to be $\deg (x) + 1$. Extend this to monomials additivitely. We can filter $\mathcal M$ by considering elements of weight at least $k$. Now $d$ increases the weight by at least $2$ because $\mathcal M$ is minimal. Further $i$ decreases the weight of an element by at most $1$. Then $di + id$ increases the weight of an element by at least $1$ while preserving the degree, but in a given degree $d$ the maximum weight possible is $2d$, so $di+ di$ is nilpotent. 
\end{proof}

We will now write $U$ for $\exp(\mathfrak u)$ and $U_*$ for $\exp(\mathfrak u_*)$. Because $\mathfrak u$ consists of nilpotent elements, the exponential of an element in $\mathfrak u$ has finitely many non-zero terms. Since elements of $\mathfrak u $ and $\mathfrak u_*$ commute with $d$, $U$ and $U_*$ are comprised of CDGA maps $\mathcal M \rightarrow \mathcal M$. Our next goal will be to show that $U$ and $U^*$ are exactly the subgroups of $\Aut(\mathcal M)$ of elements which are homotopic to the identity and augmentation preserving homotopic to the identity respectively. It is worth remarking that even though $\mathcal M$ has only one choice of augmentation, since $\mathcal M(0) \iso \Q$, this does not mean that all homotopies are augmentation preserving.    

\begin{proposition}
	Elements of $U$ are homotopic to the identity. Moreover, elements of $U_*$ are augmentation preserving homotopic to the identity.  
\end{proposition}

\begin{proof}
	Here $\Q(t,dt)$ will denote the free CDGA on $t$ and $dt$ with $d(t) = dt$ and $d(dt) = 0$ where $\deg(t) = 0$. For a derivation $i$ of degree $-1$ define the degree $-1$ derivation
	$$I \colon \mathcal M \otimes \Q(t,dt) \rightarrow \mathcal M \otimes \Q(t,dt)$$
	where for $x \in \mathcal M$ we take $I(x) = ti(x)$ and $I(t) = 0$. Then the composition $$\mathcal M \xrightarrow{\id_{\mathcal M} \otimes 1} \mathcal M \otimes \Q(t,dt) \xrightarrow{\exp([I,d_{\mathcal M \otimes \Q(t,dt)}])} \mathcal M \otimes \Q(t,dt)$$
	provides a homotopy from $\id_\mathcal M$ to $\exp([i,d])$. This shows that elements of $U$ are homotopic to the identity. 
	
	Suppose further that $i$ vanishes on degree $1$ elements. Then we must check that the above is an augmentation preserving homotopy, that is we must check that the image of the homotopy is contained in $(\overline{\mathcal M} \otimes \Q(t,dt)) \oplus \Q$, where $\overline{\mathcal M}$ is the augmentation ideal of $\mathcal M$. We will show that $[I,d_{\mathcal M \otimes \Q(t,dt)}]x$ lands in $(\overline{\mathcal M} \otimes \Q(t,dt)) \oplus \Q$ for $x$ a degree $n$ generator of $\mathcal M$. We compute $$[I, d_{\mathcal M \otimes \Q(t,dt)}]x = t i(dx) + i(x)dt + t di(x) = t(id + di)x + i(x)dt.$$ For the first term on the right-hand side $\deg([i,d]x) \neq 0$ and so $t[i,d]x \in (\overline{\mathcal M} \otimes \Q(t,dt)) \oplus \Q$. Now for the second term, if $n \neq 1$ then $\deg(i(x)) \neq 0$ so $i(x) dt \in (\overline{\mathcal M} \otimes \Q(t,dt)) \oplus \Q$. Then if $n = 1$ then $i(x) = 0$ using our assumption on $i$, so $i(x)dt \in (\overline{\mathcal M} \otimes \Q(t,dt)) \oplus \Q$. 
	\end{proof}

\begin{proposition}
	If a map $f\colon \mathcal M \rightarrow \mathcal M$ is homotopic to the identity then $f \in U$. Moreover, if $f$ is augmentation preserving homotopic to the identity then $f \in U_*$. 
\end{proposition}

\begin{proof}
	Suppose that we have a homotopy $F + Gdt \colon \mathcal M \rightarrow \mathcal M \otimes \Q(t,dt)$ such that $F|_{t=0} = \id_{\mathcal M}$. We need to show that $F|_{t = 1}$ is of the form $\exp([i,d])$. To do this we use work of Block--Lazarev \cite[p. 10--11]{Block2005}, rather than Sullivan's original strategy, because their construction is better suited to the based case. They show  
	$$F(1) = \exp\Big(\Big[\textstyle\int_0^1 GF^{-1}dt, d\Big]\Big)$$
	where the integral is formal and $F^{-1}$ is a map from $\mathcal M \otimes \Q(t,dt)$ to $\mathcal M$. This immediately gives that that any map homotopic to the identity is in $U$ (this is Block--Lazarev's motivation) by taking $i = \textstyle\int_0^1 GF^{-1}dt$. Suppose further that the homotopy $F + Gdt$ is an augmentation preserving homotopy, i.e. its image is in $(\overline{\mathcal M} \otimes \Q(t,dt)) \oplus \Q$. Since $G$ lowers the degree of an element it must vanish on the degree $1$ generators of $M$. Then for $x$ a degree $1$ generator of $\mathcal M$ we get $\left(\int_0^1GF^{-1}dt \right) (x) = 0$ because $F$ preserves degree.
\end{proof}

We can now build the algebraic group in which $\pi_0 \hAut_*(X)$ is arithmetic. Let $\Aut(\mathcal M)$ denote the group of all CDGA automorphisms and let ${\sim}$ and $\sim_*$ denote the equivalence relations of CDGA homotopy and augmentation preserving CDGA homotopy.

\begin{proposition}\label{AlgebraicityCorollary}
	$\Aut(\mathcal M) / {\sim}$ and $\Aut(\mathcal M) / {\sim}_*$ are algebraic groups. 
\end{proposition}

\begin{proof}
	To begin, suppose that $X$ is a Postnikov truncation of a finite nilpotent space. Then $\mathcal M$ is finitely generated so $\Aut(\mathcal M)$ is an algebraic group. The groups $U$ and $U_*$ are normal subgroups of $\Aut(\mathcal M)$ (conjugation preserves maps which are (augmentation preserving) homotopic to the identity). Since $U$ and $U_*$ are the exponentials of nilpotent Lie algebras, they are unipotent algebraic groups. Moreover, they are algebraic subgroups of $\Aut(\mathcal M)$ by \cref{SpecialAlgebraicGroupFact}. So, the quotients $$\Aut(\mathcal M) /U \iso \Aut(\mathcal M) / {\sim} \hspace{1cm} \Aut(\mathcal M) /U_* \iso \Aut(\mathcal M) / {\sim}_*$$ are algebraic groups \cite[\S 12.1]{Humphreys1975}.
	
	Now for $X$ which is not a Postnikov truncation, obstruction theory gives the equivalences 
	$$\Aut(\mathcal M) /{\sim} \iso \Aut(\mathcal M^n) / {\sim} \hspace{1cm} \Aut(\mathcal M) /{\sim_*} \iso \Aut(\mathcal M^n) / {\sim}_*$$
	where $n$ is the cohomological dimension of $X$ and $\mathcal M^n$ is the minimal model of the $n^{\text{th}}$ Postnikov truncation of $X$.
\end{proof}

\begin{proposition}\label{CohomologyAction}
	For $V$ a finite dimensional vector space, the actions of \linebreak $\Aut(\mathcal M)/{\sim}$ and $\Aut(\mathcal M)/{\sim}_*$ on $H^*(\mathcal M; V)$ are algebraic.
\end{proposition}

\begin{proof}
Here $H^*(\mathcal M; V) \cong H^*( \mathcal M) \otimes V$ and so we can disregard the coefficients. 

First suppose that $X$ is a Postnikov truncation of a finite nilpotent space. The map $\Aut(\mathcal M) \to \Aut(H^*(\mathcal M))$ is given by taking an automorphism, restricting to the kernel of $d$, and then taking the automorphism on quotients, all of which is algebraic (here by $\Aut(H^*(\mathcal M))$ we mean the automorphisms as a graded algebra). Finally the actions of $\Aut(\mathcal M)/{\sim}$ and $\Aut(\mathcal M)/{\sim}_*$ on $\Aut(H^*(\mathcal M))$ exist and are algebraic by the universal property of the quotient \cite[\S 12.1]{Humphreys1975} for algebraic groups.

Now for $X$ which is not a Postnikov truncation, using the Serre spectral sequence we have that $H^k(\mathcal M)$ agrees with $H^k(\mathcal M^n)$ for $k \leq n$ where $\mathcal M^n$ is the minimal model of the $n^{\text{th}}$ Postnikov truncation of $X$. Picking $n$ to be the cohomological dimension of $X$ we have a map $\Aut(H^*(\mathcal M^n)) \to \Aut(H^*(\mathcal M))$ given by setting any elements of degree higher than $n$ to $0$. This is map algebraic and so the action map given by the composition $$\Aut(\mathcal M^n)/{\sim} \cong \Aut(\mathcal M^n)/{\sim} \to \Aut(H^*(\mathcal M^n)) \to \Aut(H^*(M))$$
is also algebraic (here ${\sim}$ can be replaced with ${\sim_*}$). 
 
\end{proof}

This completes the construction of the algebraic group which we will map our based self-equivalences into. To study this map we will use the following obstruction theory result.

\begin{fact}[Nomura {\cite[Theorem 1.3]{Nomura1966}}]\label{ObstructionTheory}
Suppose we have a principal fibre sequence of based spaces, i.e. a commutative diagram
\[\begin{tikzcd}[]
	F & E & B \\
	{\Omega B'} & {F'} & {E'} & {B'}
	\arrow[from=1-1, to=1-2]
	\arrow["\simeq", from=1-1, to=2-1]
	\arrow[from=1-2, to=1-3]
	\arrow["\simeq", from=1-2, to=2-2]
	\arrow["\simeq", from=1-3, to=2-3]
	\arrow[from=2-1, to=2-2]
	\arrow[from=2-2, to=2-3]
	\arrow[from=2-3, to=2-4]
\end{tikzcd}.\]
Further suppose that $\pi_k(F) = 0$ for $k < n$ and $\pi_k(B) = 0$ for $k \geq  n$. Then there is an exact sequence  $$ \ker ([E,F] \rightarrow [F,F]) \rightarrow \pi_0\hAut_* (E) \rightarrow \pi_0\hAut_*(F) \times \pi_0 \hAut_*(B).$$
\end{fact}

With our more careful setup, we can now induct up the Postnikov tower of $X$. 

\begin{theorem}\label{FiniteKernelImageArithmetic}
	The map $$\pi_0\hAut_*(X) \to \Aut(\mathcal M)/ {\sim}_*$$
	has finite kernel and image an arithmetic group. 
\end{theorem}

\begin{proof}
Since $X$ is either a finite nilpotent space (or a Postnikov truncation thereof) there is a refinement of its Postnikov tower $$\cdots X_{n+1} \rightarrow X_n \rightarrow \cdots \rightarrow X_1 \rightarrow *$$
where each successive pair is given as the pullback of a path space fibration of an Eilenberg--Maclane space $K(\pi_{n+1}, n + 1)$ with $n$ non-strictly increasing. $\mathcal M^n$ will denote the minimal model of $X_n$. We will use $\pi_n$ to denote $n^{\text{th}}$ homotopy group of the fibre of $X_{n} \rightarrow X_{n-1}$, and we will use $Q^n$ to denote the dual rationalization.

To begin, for any $n$ the map $$\pi_0 \Hom(X_n, X_n) \cong \pi_0 \Hom(X, X_n) \rightarrow \pi_0 \Hom(X_, (X_n)_\Q) \cong \pi_0 \Hom((X_n)_\Q, (X_n)_\Q)$$ is finite-to-one \cite[Corollary II.5.4]{Hilton11975} (here the Hom-spaces are based). \linebreak The map $\pi_0(\hAut_*(X_n)) \to \Aut(\mathcal M^{n})/{\sim}_*$ is a restriction and thus has finite kernel.

We must show that the image of the map $\pi_0(\Aut(X_n)) \to \Aut(\mathcal M^{n})/{\sim}_*$ is an arithmetic group. For this will use an inductive obstruction theory argument. We note that $ \Q$ is the minimal model of the point, and $\Aut(\Q) / {\sim}_* $ is trivial, and the trivial subgroup of the trivial group is arithmetic.

Now we consider the fibration $$K(\pi_{n}, n) \rightarrow X_{n} \rightarrow X_{n-1}.$$
Since this is a principal fibration so we can apply \cref{ObstructionTheory}. We see that $$\ker([X_{n}, K(\pi_{n}, n)] \rightarrow [K(\pi_{n}, n), K(\pi_{n}, n)])$$
is the kernel of the map from the exact sequence of a pair $(X_{n}, X_{n-1})$ which gives $$\ker (H^{n}(X_{n}; \pi_{n}) \rightarrow H^{n+1}(X_{n}, X_{n-1}; \pi_{n})) \iso H^{n}(X_{n-1}; \pi_{n}).$$
So we have the exact sequence
$$H^{n}(X_{n-1};\pi_{n}) \rightarrow \pi_0 \hAut_* (X_{n}) \xrightarrow[(F,\phi)]{} \pi_0 \hAut_* (X_{n-1}) \times \Aut(\pi_{n}).$$

Since all our reasoning is still true rationally so we obtain the following comparison of exact sequences
\[\begin{tikzcd}
	{H^{n}(X_{n-1}; \pi_{n})} & {\pi_0\hAut_*(X_{n})} & \pi_0\hAut_*(X_{n-1}) \times  \Aut(\pi_{n})  \\
	{H^{n}(\mathcal M^{n-1}; Q^{n})} & {\Aut(\mathcal M^{n})/{\sim}_*} & \Aut(\mathcal M^{n-1})/{\sim}_* \times \Aut(Q^{n})
	\arrow[from=1-1, to=1-2]
	\arrow[from=1-1, to=2-1]
	\arrow["{(F,\phi)}"', from=1-2, to=1-3]
	\arrow[from=1-2, to=2-2]
	\arrow[from=1-3, to=2-3]
	\arrow[from=2-1, to=2-2]
	\arrow["({F_\Q,\phi_\Q)}"', from=2-2, to=2-3]
\end{tikzcd}.\]

We aim to modify this diagram to apply \cref{ExactSequenceAlgGroups} to conclude that the image of the middle map is an arithmetic group.

First, $\pi_0\hAut( X_{n-1})/ \times \Aut(\pi_{n})$ acts on $H^{n+1}(X_{n-1}; \pi_{n})$ by having the first term act on the space and the second term act on the coefficients. The image of $(F,\phi)$ is exactly those based self-equivalences of $X_{n-1}$ and automorphisms of $\pi_{n}$ that can be lifted to based self-equivalences of $X_{n}$, i.e. based self-equivalences of $X_{n-1}$ and automorphisms of $\pi_{n}$ that stabilize the $k$-invariant $k_n \in H^{n+1}(X_{n-1}; \pi_{n})$. This means that $\im(\F, \phi) = \stab (k_n)$. All of this is also true rationally and we have a map $\stab (k_n) \to \stab ((k_n)_\Q)$. \cref{CohomologyAction} says  that rationally the action map is algebraic and so we see that $\stab ((k_n)_\Q)$ is algebraic subgroup (see \cite[p. 52]{Borel1991}). We can view the stabilizer as the the kernel of the evaluation map at $k_n$ $\pi_0\hAut( X_{n-1})/ \times \Aut(\pi_{n}) \to H^{n+1}(X_{n-1}; \pi_{n})$. This is notably not a homomorphism but it is structured enough to apply to apply \cref{KernelAlgGroups} (here the appropriate base point for the set $H^{n+1}(X_{n-1}; \pi_{n})$ is $k_n$). This shows that the map $\stab (k_n) \to \stab ((k_n)_\Q)$ has finite kernel and image an arithmetic group. So we have reduced to the comparison of exact sequences
\[\begin{tikzcd}
	{H^{n}(X_{n-1}; \pi_{n})} & {\pi_0\hAut_*(X_{n})} & {\im(F,\phi)} & 1 \\
	{H^{n}(\mathcal M^{n-1}; Q^{n})} & {\Aut(\mathcal M^{n})/{\sim}_*} & {\im(F_\Q, \phi_\Q)} & 1
	\arrow[from=1-1, to=1-2]
	\arrow[from=1-1, to=2-1]
	\arrow["{(F,\phi)}"', from=1-2, to=1-3]
	\arrow[from=1-2, to=2-2]
	\arrow[from=1-3, to=1-4]
	\arrow[from=1-3, to=2-3]
	\arrow[from=2-1, to=2-2]
	\arrow["{(F_\Q,\phi_\Q)}"', from=2-2, to=2-3]
	\arrow[from=2-3, to=2-4]
\end{tikzcd}\]
in which the bottom right horizontal map is algebraic and the right vertical map has finite kernel and image an arithmetic subgroup.

Now we can restrict to kernels. The map $\ker (F,\phi) \to \ker(F_\Q, \phi_\Q)$ has finite kernel because the middle vertical map above does. The map $H^n(\mathcal M^{n-1};Q^n) \to \ker (F_\Q, \phi_\Q)$ is algebraic by \cref{SpecialAlgebraicGroupFact} and thus the image of the composition $$ H^n(X_{n-1}; \pi_n) \to H^n(\mathcal M^{n-1};Q^n) \to \ker (F_\Q, \phi_\Q)$$
is an arithmetic group by \cref{GeneralAlgebraicGroupFacts}. This image is exactly the image of the map $\ker (F,\phi) \to \ker(F_\Q, \phi_\Q)$. So we can restrict to the diagram
\[\begin{tikzcd}
	1 & {\ker (F,\phi)} & {\pi_0\hAut_*(X_{n})} & {\im(F,\phi)} & 1 \\
	1 & {\ker(F_\Q, \phi_\Q)} & {\Aut(\mathcal M^{n})/{\sim}_*} & {\im(F_\Q, \phi_\Q)} & 1
	\arrow[from=1-1, to=1-2]
	\arrow[from=1-2, to=1-3]
	\arrow[from=1-2, to=2-2]
	\arrow["{(F,\phi)}"', from=1-3, to=1-4]
	\arrow[from=1-3, to=2-3]
	\arrow[from=1-4, to=1-5]
	\arrow[from=1-4, to=2-4]
	\arrow[from=2-1, to=2-2]
	\arrow[from=2-2, to=2-3]
	\arrow["{(F_\Q,\phi_\Q)}"', from=2-3, to=2-4]
	\arrow[from=2-4, to=2-5]
\end{tikzcd}\]
to which we can apply \cref{ExactSequenceAlgGroups}. This completes the proof for any Postnikov truncation of a finite nilpotent space $X$. To complete the proof, we use that  for $n$ at least the homological dimension of $X$ $$\pi_0\hAut_*(X) \iso \pi_0\hAut_* (X_{n}).$$
\end{proof}

Finally, in order to conclude \cref{ArithmeticityHomotopyAutomorphisms} from \cref{FiniteKernelImageArithmetic} we need to know that for finite nilpotent $X$  the group $\pi_0 \Aut(X)$ is residually finite. This is a result of Roitberg \cite[Theorem 3]{Roitberg1984}. We can now prove the result Sullivan originally stated. 

\begin{theorem}\label{SullivanOriginalTheorem}
	For a finite nilpotent space $X$, the group $\pi_0\hAut(X)$ is commensurable to an arithmetic group. 
\end{theorem}

\begin{proof}
	If $X$ has multiple path components then $$\pi_0\Aut(X) \simeq S \ltimes \prod_{i \in \pi_0(X)} \pi_0\hAut(X_i) $$
	where $S$ is a product of symmetric groups, one of order $k$ for each cardinality $k$ collection of homotopic path components of $X$. Then $\prod_{i \in \pi_0(X)} \pi_0 \hAut(X_i)$ is finite index inside this product, and so if each $ \pi_0\Aut(X_i)$ is commensurable to an arithmetic group the $\pi_0\hAut(X)$ is commensurable to an arithmetic group. So we can assume that $X$ has only one path component.
	
Now, pick an arbitrary basepoint for $X$. Then, in the notation of \cref{FiniteKernelImageArithmetic}, we have a comparison of the exact sequences which relate the based and unbased cases
\[\begin{tikzcd}
	{\pi_1} & {\pi_0\hAut_*(X)} & {\pi_0\hAut(X)} & 1 \\
	{Q^1} & {\Aut(\mathcal M)/{\sim_*}} & {\Aut(\mathcal M)/{\sim}} & 1
	\arrow["\boundry", from=1-1, to=1-2]
	\arrow[from=1-1, to=2-1]
	\arrow[from=1-2, to=1-3]
	\arrow[from=1-2, to=2-2]
	\arrow[from=1-3, to=1-4]
	\arrow[from=1-3, to=2-3]
	\arrow["{\boundry_\Q}"', from=2-1, to=2-2]
	\arrow[from=2-2, to=2-3]
	\arrow[from=2-3, to=2-4]
\end{tikzcd}.\]
We aim to show that the right-hand map has finite kernel and image an arithmetic group. 

First, the image of the right-hand map is the image of the composition
$${\pi_0\hAut_*(X)} \to  {\Aut(\mathcal M)/{\sim_*}} \to {\Aut(\mathcal M)/{\sim}}.$$
The first map has image an arithmetic group by \cref{ArithmeticityHomotopyAutomorphisms} and the second map is surjective algebraic so the total image is arithmetic subgroup. This shows that the right most vertical map above has image an arithmetic group. 

We now show that any element $x$ be in the kernel of the right-hand map is torsion. Then let $y$ be a lift of $x$ along the map ${\pi_0\hAut_*(X)} \to {\pi_0\hAut(X)}$. Then let $\tilde y$ be the image of $y$ along the middle vertical map. Then $\tilde y$ maps to $0$ along the map ${\Aut(\mathcal M)/{\sim_*}} \to {\Aut(\mathcal M)/{\sim}}$ so it has some lift $\tilde z$ along $\boundry_\Q$. Now the map $\pi_1 \to Q^1$ is given by the Mal'cev completion \cite{Malcev1949}, in other words, this is the map from $\pi_1$ to its uniquely divisible closure. This tells us that for some $n$ there is a lift $w$ of $\tilde z^n$ in $\pi_1$. Now $\boundry w$ and $y^n$ both map to $\tilde y^n$ via ${\pi_0\hAut_*(X)} \to {\Aut(\mathcal M)/{\sim_*}}$ and thus differ by something in the kernel this map, i.e. $y^n = t \boundry w$ for some torsion element $t$. $\boundry w$ maps to $1$ along the map ${\pi_0\hAut_*(X)} \to {\pi_0\hAut(X)}$ and thus $x^n$ is a torsion element meaning that $x$ itself is torsion.

Now, to avoid issues with the Burnside problem, we must establish some kind of finiteness property for kernel of the right-hand map. Let $\tilde \Gamma$ denote the intersection of the image of $\boundry_\Q$ with the image of ${\pi_0\hAut_*(X)} \to {\Aut(\mathcal M)/{\sim_*}}$ and let $\Gamma$ denote the preimage of $\tilde \Gamma$ under ${\pi_0\hAut_*(X)} \to {\Aut(\mathcal M)/{\sim_*}}$. Now $\tilde \Gamma$ is subgroup of a unipotent algebraic group and is thus nilpotent. This means that $\Gamma$ is finite-by-nilpotent, and hence polycyclic-by-finite \cite[Proposition 1.A.2]{Segal1983}. Now the image of $\Gamma$ under the map ${\pi_0\hAut_*(X)} \to {\pi_0\hAut(X)}$ is exactly the kernel we want to show is finite. Thus this kernel is a quotient of a polycyclic-by-finite group and is therefore polycyclic-by-finite. 

So the kernel of the right-hand map is a torsion polycyclic-by-finite group which we can see to just be finite (the subnormal series for the polycyclic part shows that the group is built via extensions from a finite number of finite groups).

This shows that the map ${\pi_0\hAut(X)} \to {\Aut(\mathcal M)/{\sim}}$ has finite kernel and image an arithmetic group. Then by work of Mezher \cite[Theorem E]{Mezher2024} we know that $\pi_0\hAut(X)$ is residually finite (Mezher's result is stated only for simply connected spaces, but their proof applies without issue to nilpotent spaces). Thus, ${\pi_0\hAut(X)}$ is commensurable to an arithmetic group.
\end{proof}


\bibliography{bibliography.bib}
\bibliographystyle{mainbibstyle}

\end{document}